\documentclass[10pt,a4paper]{scrartcl}
\usepackage{a4wide}

\usepackage[utf8]{inputenc}
\usepackage[T1]{fontenc}
\usepackage[english]{babel}
\usepackage{csquotes}

\usepackage[pdftex]{graphicx}
\usepackage{latexsym}
\usepackage{amsmath,amssymb,amsthm}
\usepackage{dsfont}
\usepackage{dlfltxbcodetips}
\usepackage{mathrsfs}
\usepackage{mathtools}
\usepackage{tikz}
\usepackage{float}
\usepackage{enumitem}
\usepackage{hyperref}
\usepackage{cleveref}
\usepackage[marginpar]{todo}
\usepackage{subcaption}

\newtheorem{theorem}{Theorem}
\newtheorem{corollary}[theorem]{Corollary}

\newtheorem{lemma}[theorem]{Lemma}
\newtheorem{proposition}[theorem]{Proposition}

\theoremstyle{definition}

\newtheorem*{Example*}{Example}

\newtheorem{remark}[theorem]{Remark}

\newtheorem*{question*}{Question}

\theoremstyle{remark}

\newtheorem*{Remark_small*}{Remark}

\numberwithin{equation}{section}

\newcommand{\R}{\mathbb{R}} 
\newcommand{\N}{\mathbb{N}} 
\newcommand{\Hy}{\mathbb{H}}
\newcommand{\BB}{\mathbb{B}}

\renewcommand{\P}{\mathbb{P}}
\newcommand{\V}{\mathbb{V}}

\newcommand{\JJ}{J}

\DeclareMathOperator{\dint}{d\!}

\DeclareMathOperator{\vol}{vol}
\DeclareMathOperator{\Beta}{Beta}

\makeatletter
\let\@fnsymbol\@alph
\makeatother

\setkomafont{title}{\normalfont\Large}
\setkomafont{author}{\normalfont\large}

\begin{document}
\title{\textbf{High-dimensional limits arising from \\ hyperbolic Poisson $k$-plane processes}}
\author{Tillmann B\"uhler\footnotemark[1],\;\; Daniel Hug\footnotemark[2],\;\; Christoph Th\"ale\footnotemark[3]}
\date{}
\maketitle
\renewcommand{\thefootnote}{\fnsymbol{footnote}}
\footnotetext[1]{Karlsruhe Institute of Technology, Germany. Email: tillmann.buehler@kit.edu}
\footnotetext[2]{Karlsruhe Institute of Technology, Germany. Email: daniel.hug@kit.edu}
\footnotetext[3]{Ruhr University Bochum, Germany. Email: christoph.thaele@rub.de}

\vspace{-1cm}

\begin{abstract}
\noindent
We consider a stationary Poisson process of $k$-planes in the $d$-dimensional hyperbolic space $\mathbb H^d$ of constant curvature $-1$, with $d \ge 4$ and $1 \le k \le d-1$. 
It is known that, after centring and normalization, the total $k$-volume of all intersections of $k$-planes with a geodesic ball of radius $R$ converges in distribution, as $R \to \infty$, to a non-Gaussian infinitely divisible random variable $Z_{d,k}$ whenever $2k > d+1$. 
We investigate the distributional behaviour of $Z_{d,k}$ in the high-dimensional regime $d \to \infty$ and depending on how fast $k$ grows in relation to $d$. 
We derive precise conditions for the variance normalized sequence to converge in law to a standard Gaussian random variable or to a degenerate law, respectively, and show that an alternative rescaling of the Lévy measures yields an explicit non-Gaussian infinitely divisible limit for fixed codimension $d-k$ and a standard Gaussian limit for $d-k \to \infty$.
\\

\end{abstract}

\section{Introduction}

In stochastic geometry, a central theme is the analysis of geometric structures generated by random point processes on spaces of geometric objects such as balls, general convex bodies or subspaces. Of particular interest are their geometric characteristics, including volumes, surface areas, and more general intersection measures. While much of the classical theory has been developed in the Euclidean setting, see \cite{SW08}, there has recently been growing interest in extending these questions to spaces of constant curvature. In particular, in the hyperbolic space $\Hy^d$, with its constant negative curvature and exponential volume growth, such random geometric structures often exhibit an interesting probabilistic behaviour that differs significantly from that in Euclidean space.  

One class of models that has received considerable attention is that of {Poisson processes of $k$-planes} in $\Hy^d$, where a $k$-plane is a totally geodesic submanifold of dimension $k\in\{1,\ldots,d-1\}$, see \cite{BHT23,BH25,HHT21,KRT25} and the final chapter of \cite{HS24}. To be formal, let $\eta_{d,k}$ be an isometry invariant Poisson process on the space of $k$-planes in $\Hy^d$. Let $\BB^d_R$ denote a geodesic ball of radius $R>0$ centred at a fixed point in $\Hy^d$, and define  
\[
F_{d,k,R} := \sum_{E\in\eta_{d,k}} \vol_k(E\cap\BB_R^d),
\]
where $\vol_k(\,\cdot\,)$ denotes the $k$-dimensional Riemannian volume in $E$. In other words, $F_{d,k,R}$ measures the total $k$-dimensional volume contributed by all $k$-planes of $\eta_{d,k}$ inside of $\BB_R^d$.  

In the Euclidean setting, for all choices of dimensions $d\geq 2$ and $k\in\{1,\ldots,d-1\}$ the random variables $F_{d,k,R}$ satisfy a central limit theorem as $R\to\infty$ under centring and variance normalization. In hyperbolic space, however, this asymptotic normality fails whenever $2k>d+1$, see \cite{HHT21} for the special case $k=d-1$ and \cite{BHT23} for general $k$. In this dimension regime, it has been shown that  $F_{d,k,R}$ converges in distribution, after centring and normalization, to a centred {non-Gaussian infinitely divisible} random variable $Z_{d,k}$.
The characteristic function of $Z_{d,k}$ is given by
\begin{equation}\label{eq:char_fct_Z}
    \varphi_{Z_{d,k}}(t) = \exp\Big( \int_\R (e^{itx}-1-itx) \,\nu_{d,k}(\dint x) \Big), \quad t \in \R,
\end{equation}
with Lévy measure (the terms \emph{Lévy measure} and \emph{Lévy triplet} are defined in Section \ref{sec:Background})
\begin{equation}\label{eq:levy_measure_Z}
	\nu_{d,k}(\dint x)
	=\frac{\omega_{d-k}}{k-1}\;x^{-1-\frac{d-1}{k-1}}
	\bigl(1-x^{\frac{2}{k-1}}\bigr)^{\frac{d-k}{2}-1}\,
	\mathbf 1_{(0,1)}(x)\,\dint x,
\end{equation}
where $\omega_b=2\pi^{b/2}/\Gamma(b/2)$ for $b\in\N$ is the surface area of a Euclidean sphere in $\R^b$.
Note that the variance $\sigma_{d,k}^2 := \V(Z_{d,k}) = \int_\R x^2\,\nu_{d,k}(\dint x)$ is finite but not equal to $1$ in general, see \eqref{eq:sigma} below.
The above normalization was chosen in \cite{BHT23} to obtain a  representation with $\nu_{d,k}$ concentrated in $(0,1)$.

The same limit law also arises in the study of more general intersection measures for hyperbolic Poisson $k$-plane processes, as demonstrated in \cite{BH25}. Motivated by simulations in \cite[Section 6]{BH25}, it has been asked whether (or not), in the high-dimensional limit $d\to\infty$ (with $2k>d+1$), the variance normalized random variables $Z_{d,k}$ converge to a standard Gaussian one. Surprisingly, the cumulant analysis in \cite{BH25} reveals a highly unconventional behaviour. After normalization, the higher-order cumulants of \( Z_{d,k} \) do not merely fail to vanish, they even diverge as \( d \to \infty \). This stands in stark contrast to the standard central limit setting, where vanishing higher cumulants are the hallmark of Gaussian convergence. The result therefore suggests that \( Z_{d,k} \) should exhibit a non-Gaussian limit, if any. Remarkably, however, we will see that this intuition is misleading. Despite the divergence of its cumulants, \( Z_{d,k} \) can in fact converge to a Gaussian law. Our first theorem identifies the precise conditions under which this counter-intuitive phenomenon occurs.

\begin{theorem}\label{thm:CLT}
    Let $(d_n,k_n)_{n\geq 1}$ be a sequence satisfying $k_n < d_n$, $2k_n > d_n+1$ and $d_n \to \infty$ as $n\to\infty$.
    Write $Z^*_n \coloneqq Z_{d_n,k_n}/\sigma_{d_n,k_n}$ for the variance normalized version of $Z_{d_n,k_n}$ and abbreviate $r_n \coloneqq 2k_n - d_n - 1$.
    \begin{enumerate}
        \item[\emph{(a)}] If $k_n/d_n \to 1/2$ and $\limsup\limits_{n\to \infty} d_n^{-1} \, r_n^{d_n/k_n} < e\pi$, then $Z^*_n$ converges in distribution to a standard Gaussian law.
         \item[\emph{(b)}] If  $k_n/d_n \to 1/2$ and $\liminf\limits_{n\to \infty} d_n^{-1} \, r_n^{d_n/k_n} > e\pi$ or if $\liminf\limits_{n \to \infty}k_n/d_n > 1/2$, then $Z^*_n$ converges in distribution to $0$.
    \end{enumerate}
\end{theorem}

\begin{remark}
Let $\gamma>0$ and $\beta\in (0,1)$ be fixed and $\tau(d_n)\in (0,1)$ such that  $k_n=\frac{d_n}{2}+\gamma d_n^\beta+\tau(d_n)$ is an integer. In this case, we get
$$
d_n^{-1}r^{{d_n}/{k_n}}\sim 4\gamma^2 d_n^{2\beta-1}\to \begin{cases}
    0,&  \beta\in (0,1/2),\\
    (2\gamma)^{2},& \beta=1/2,\\
    \infty,&\beta\in (1/2,1),
\end{cases}
$$
where for two sequences $(a_n)_{n\geq 1}$ and $(b_n)_{n\geq 1}$ we write $a_n\sim b_n$ if $a_n/b_n\to 1$ as $n\to\infty$.
Hence, for $\beta=1/2$, if $\gamma<\frac{1}{2}\sqrt{e\pi}$, then  $Z_n^*\to \mathcal{N}(0,1)$ in distribution, and if  $\gamma>\frac{1}{2}\sqrt{e\pi}$, we get $Z_n^*\to 0$ in distribution. At the critical growth rate with $\beta=1/2$ and $2\gamma =\sqrt{e\pi}$, the asymptotic behaviour of $Z_n^*$ remains open.  
\end{remark}

Our second goal is to show that even when convergence to a normal distribution fails, a different type of scaling can still lead to a meaningful limit law.
Instead of considering $Z_{d,k}/\sigma_{d,k}$, we rescale the Lévy measure $\nu_{d,k}$ and study the random variable $\widetilde Z_{d,k} $ 
with Lévy triplet $(0,0,\nu_{d,k}/\sigma^2_{d,k})$.
Such a rescaling and associated random variables are considered when introducing infinitely divisible distributions and their Lévy triplets, see \cite[\S 17, Theorem 4]{GK54},  \cite[Theorem 7.10,  Corollary 8.3]{Sato} and Section \ref{sec:Background}.

It turns out that when the codimension $d_n-k_n = b \in \N$ is fixed, the random variables $\widetilde Z_{d_n,k_n}$ converge in distribution to some non-Gaussian limit $\widetilde Z^{(b)}$ (see Figure \ref{fig:Density} for some illustrations).
On the other hand, when $d_n-k_n \to \infty$, then $\widetilde Z_{d_n,k_n}$ converges to a standard Gaussian distribution.

In the following, we write `$\log$' for the natural logarithm with respect to base $e$.

\begin{theorem}\label{thm:rescaling_levy_measure}
Let $(d_n,k_n)_{n\geq 1}$ be a sequence satisfying $k_n < d_n$, $2k_n > d_n+1$ and $d_n \to \infty$ as $n\to\infty$.
Let $(0,0,\nu_{d_n,k_n})$ denote the Lévy triplet of $Z_{d_n,k_n}$ and let $\widetilde Z_{n}$ be a random variable with Lévy triplet $(0,0,\nu_{d_n,k_n}/\sigma^2_{d_n,k_n})$.
\begin{enumerate}[label=\emph{(\alph*)}]
    \item If the codimension $d_n-k_n = b \in \N$ is constant, then $\widetilde Z_n$ converges in distribution to an infinitely divisible random variable $\widetilde Z^{(b)}$ with Lévy triplet $(0,0,\widetilde\nu^{(b)})$, where
    \begin{equation*}
        \widetilde\nu^{(b)}(\dint x) = \Gamma({\textstyle\frac{b}{2}})^{-1} x^{-2} \bigl(-\log x\bigr)^{\frac{b-2}{2}}\;\mathbf 1_{(0,1)}(x)\,\dint x.
    \end{equation*}
    In particular, $\widetilde Z^{(b)}$ has finite moments of all orders and $\V(\widetilde Z^{(b)}) = 1$.
    \item If $d_n-k_n \to \infty$, then $\widetilde Z_n$ converges in distribution to a standard Gaussian law.
\end{enumerate}
\end{theorem}

\begin{figure}[t]
\centering
\includegraphics[width=0.32\columnwidth]{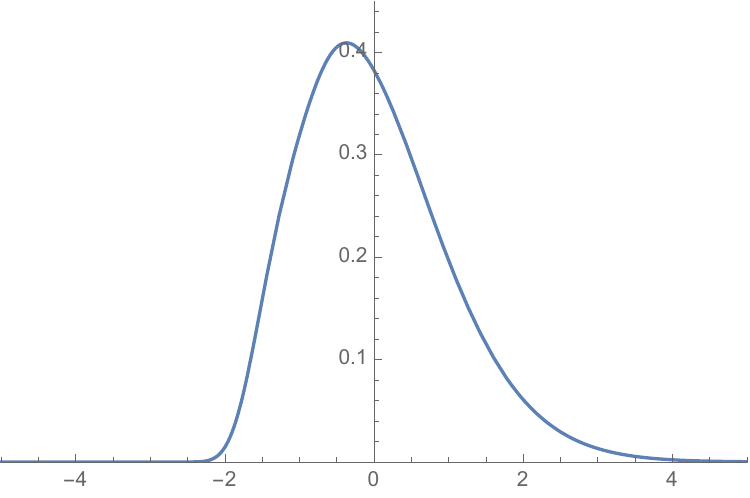}
\includegraphics[width=0.32\columnwidth]{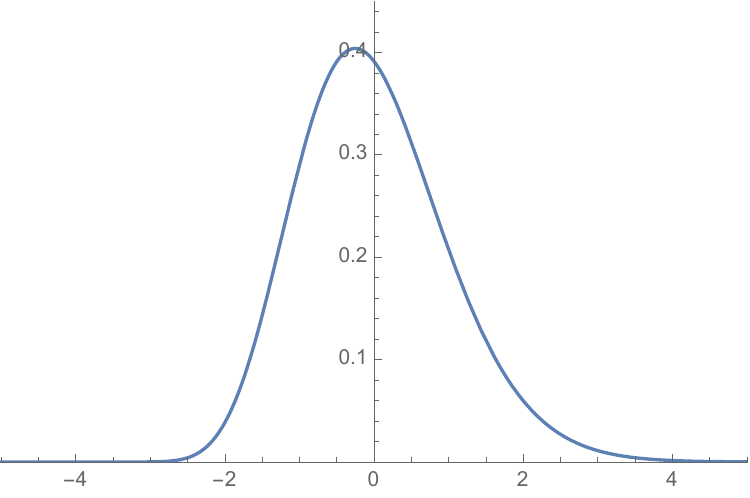}
\includegraphics[width=0.32\columnwidth]{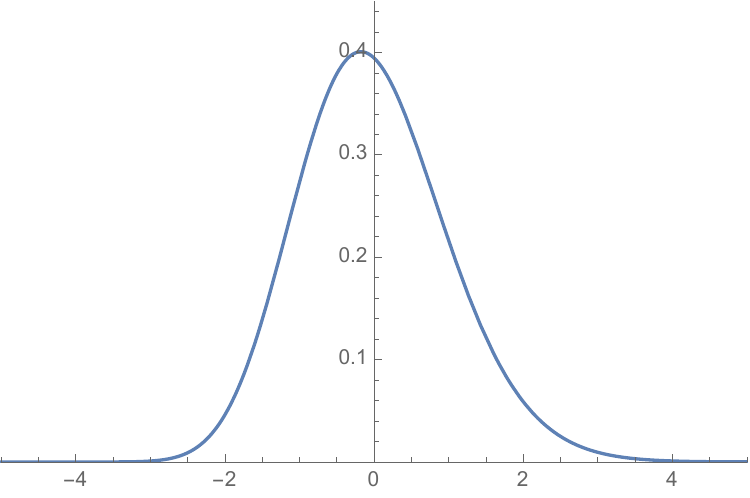}
\includegraphics[width=0.32\columnwidth]{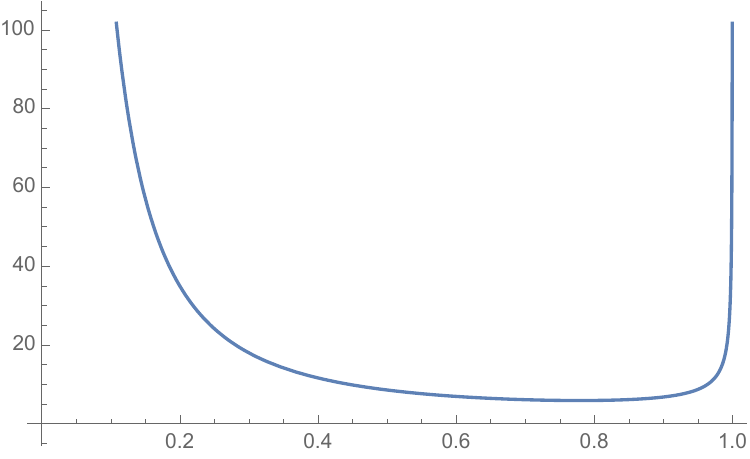}
\includegraphics[width=0.32\columnwidth]{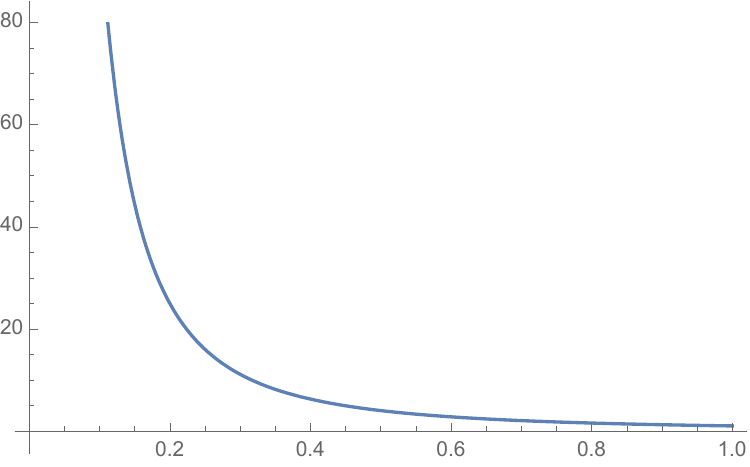}
\includegraphics[width=0.32\columnwidth]{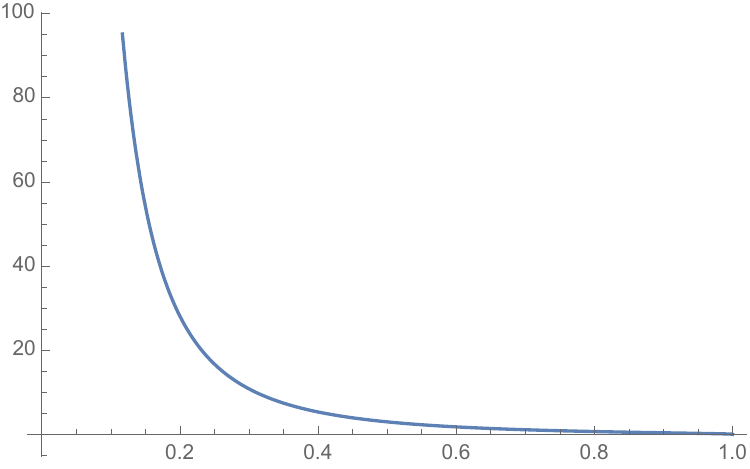}
\caption{Densities of the random variables $\widetilde Z^{(b)}$ (top line) and densities of the Lévy measure $\nu^{(b)}_*$ (bottom line) with $b=1$ (left), $b=2$ (middle) and $b=3$ (right).
The probability densities are obtained in a manner similar to that described in Section 6 of \cite{BH25}.}
\label{fig:Density}
\end{figure}

\smallskip  

The remainder of this note is organized as follows. Section~\ref{sec:Background} recalls background material on infinitely divisible laws and introduces the main convergence tool used in the proof of Theorem~\ref{thm:CLT}. In Section~\ref{sec:Beta}, we collect key inequalities and asymptotic results for the beta, incomplete beta, and regularized beta functions, which form the second main ingredient for Theorem~\ref{thm:CLT}. The proofs of Theorems~\ref{thm:CLT} and~\ref{thm:rescaling_levy_measure} are given in Section~\ref{sec:proofs}.

\section{Background on infinitely divisible laws}\label{sec:Background}

A probability distribution $\mu$ on $\mathbb{R}$ is called \emph{infinitely divisible} if, for every $n\in\N$, there exists a probability distribution $\mu_n$ such that $\mu = \mu_n^{*n}$, where $\mu_n^{*n}$ is the $n$-fold convolution of $\mu_n$ with itself. A random variable $X$ is called infinitely divisible if its distribution $\mu$ satisfies this property. The class of infinitely divisible distributions is characterized by the \emph{Lévy--Khintchine formula}. It says that the characteristic function $\varphi_X(t):=\int_{\mathbb{R}} e^{itx}\,\mu(\dint x)$ of an infinitely divisible random variable $X$ with law $\mu$ admits the unique representation
\[
\varphi_X(t) = \exp\Big( i\gamma t - \frac{1}{2}a t^2 + \int_{\R} \left( e^{itx} - 1 - itx\,\mathbf{1}_{\{|x|\le 1\}}\right) \nu(\dint x) \Big), \quad t\in\mathbb{R},
\]
see \cite[Corollary 7.6]{kallenberg}.
Here, $\gamma \in \R$ is called the drift term, $a \geq 0$ is the Gaussian variance coefficient, and $\nu$ is the \textit{Lévy measure}, satisfying $\nu(\{0\})=0$ and $\int_\R (1 \wedge x^2)\,\nu(\dint x) < \infty$.
The triplet $(\gamma,a,\nu)$, which we call the \emph{Lévy triplet} of $X$, is unique.
Convergence of infinitely divisible distributions can be described only in terms of the convergence of their associated Lévy triplets. In fact, the class of infinitely divisible distributions is closed under weak convergence, see \cite[Theorem 7.7 (i)]{kallenberg}.

The second moment of an infinitely divisible random variable $X$ with Lévy measure $\nu$ is finite if and only if $\int_\R x^2 \,\nu(\dint x) < \infty$.
In this case, the above representation can be simplified to
\begin{equation}\label{eq:Kolmogorov_rep}
    \varphi_X(t) = \exp\Big( i\tilde\gamma t - \frac{1}{2}a t^2 + \int_{\R} \left( e^{itx} - 1 - itx\right) \nu(\dint x) \Big), \quad t\in\mathbb{R},
\end{equation}
which is the so-called \emph{Kolmogorov representation}, cf.\ \cite[Equation (18.10)]{GK54} or \cite[Theorem 4.1.2]{LR79}, although the Lévy-Khintchine formula and the Kolmogorov representation are parametrized somewhat differently there.
Note that in general, $\gamma$ is not equal to $\tilde\gamma$, as one needs to adjust for the missing indicator term $\mathbf{1}_{\{|x|\le 1\}}$.

For our purposes, it suffices to treat infinitely divisible distributions whose characteristic functions are of the form \eqref{eq:Kolmogorov_rep}, with $\tilde\gamma=0$ and $a = 0$.
For these, we now state and prove a simple convergence criterion.
While the proof can be derived from a deeper result about the convergence of infinitely divisible laws (see, e.g., \cite[Theorem 7.7]{kallenberg}), we prefer to give a simple direct argument here to keep the paper self-contained. In what follows, we indicate convergence in distribution by the symbol $\xrightarrow{D}$ and write $\mathcal{N}(0,1)$ for a standard Gaussian random variable. 

\begin{lemma}\label{lem:condition_CLT_alternative}
    Let \((X_n)_{n\geq 1}\) be a sequence of random variables.
    Assume that their characteristic functions are of the form
    \[ \varphi_{X_n}(t) = \exp(\Psi_n(t)) = \exp\Big(\int_\R (e^{itx} - 1 - itx) \,\nu_n(\dint x)\Big), \quad t \in \R,\]
    with measures \(\nu_n\) on $\R$ that satisfy
    \begin{equation}\label{eq:lem_CLT_condition_1}
        \int_\R x^2 \,\nu_n(\dint x) = 1.
    \end{equation} 
    \begin{enumerate}[label=\emph{(\alph*)}]
    \item If
    \begin{equation}\label{eq:lem_CLT_condition_2}
        \int_{\{|x| > \varepsilon\}} x^2 \,\nu_n(\dint x) \to 0 \quad \text{for all} \quad \varepsilon > 0,
    \end{equation}
    then $X_n \xrightarrow{D}\mathcal{N}(0,1)$.
    \item If
    \begin{equation}\label{eq:lem_CLT_condition_3}
        \int_{\{|x| \leq \varepsilon\}} x^2 \,\nu_n(\dint x) \to 0 \quad \text{for all} \quad \varepsilon > 0,
    \end{equation}
    then $X_n \xrightarrow{D} 0$.
    \end{enumerate}
\end{lemma}
\begin{proof}
    We recall the following bound for the Taylor approximation of the exponential function from \cite[Lemma 6.15]{kallenberg}:
    \begin{equation}\label{eq:bound_taylor}
        \Big|e^{ix} - \sum_{k=0}^n \frac{(ix)^k}{k!}\Big| \leq \min\Big\{\frac{2|x|^n}{n!},\frac{|x|^{n+1}}{(n+1)!}\Big\},\quad x\in\R,\,n\in\N.
    \end{equation}
    
    To prove (a),
    assume that \eqref{eq:lem_CLT_condition_2} holds.
    Note that by Lévy's continuity theorem, it is enough to show that $\Psi_n(t) \to -t^2/2$ for all $t \in \R$.
    We thus fix $t \in \R$ and observe that by \eqref{eq:lem_CLT_condition_1} and the triangle inequality,
    \[ |\Psi_n(t) - (-t^2/2)| = \left| \Psi_n(t) - \int_\R \frac{(itx)^2}{2} \,\nu_n(\dint x) \right| \leq \int_\R \left|e^{itx} - 1 - itx - \frac{(itx)^2}{2}\right| \,\nu_n(\dint x), \]
    for all $n\in\N$.
    Splitting the integral at an arbitrary $\varepsilon > 0$ and applying \eqref{eq:bound_taylor} to both parts, we obtain
    \begin{align*}
    &\int_\R \left|e^{itx} - 1 - itx - \frac{(itx)^2}{2}\right| \,\nu_n(\dint x) 
    \leq \int_{\{|x|\leq\varepsilon\}} \frac{|tx|^3}{6} \,\nu_n(\dint x) + \int_{\{|x| > \varepsilon\}} |tx|^2 \,\nu_n(\dint x)\\
    &\leq \frac{|t^3|}{6} \int_{\{|x|\leq\varepsilon\}} \varepsilon x^2 \,\nu_n(\dint x) + t^2 \int_{\{|x| > \varepsilon\}} x^2 \,\nu_n(\dint x)
    \leq \varepsilon \frac{|t^3|}{6} + t^2 \int_{\{|x| > \varepsilon\}}x^2 \,\nu_n(\dint x),
    \end{align*}
    where we used \eqref{eq:lem_CLT_condition_1} in the last step.
    Now \eqref{eq:lem_CLT_condition_2} yields
    \[ \limsup_{n\to\infty} |\Psi_n(t) - (-t^2/2)| \leq \varepsilon |t|^3/6 + 0. \]
    Since $\varepsilon$ was arbitrary, the first claim follows.

    For part (b), we assume that \eqref{eq:lem_CLT_condition_3} holds. Let again $t\in\R$ be fixed. 
    By the triangle inequality, for all $n\in\N$ we get that
    \[ |\Psi_n(t)| \leq \int_\R |e^{itx} - 1 - itx| \,\nu_n(\dint x). \]
    Splitting the integral at an arbitrary $\varepsilon > 0$ and applying \eqref{eq:bound_taylor} to both parts, we obtain
    \begin{align*}
    &\int_\R |e^{itx} - 1 - itx| \,\nu_n(\dint x) 
    \leq \int_{\{|x|\leq\varepsilon\}} \frac{|tx|^2}{2} \,\nu_n(\dint x) + \int_{\{|x| > \varepsilon\}} 2|tx| \,\nu_n(\dint x)\\
    &\leq \frac{t^2}{2} \int_{\{|x|\leq\varepsilon\}} x^2 \,\nu_n(\dint x) + 2|t| \int_{\{|x| > \varepsilon\}} \frac{x^2}{\varepsilon}\,\nu_n(\dint x)
    \leq \frac{t^2}{2} \int_{\{|x|\leq\varepsilon\}} x^2 \,\nu_n(\dint x) + \frac{2|t|}{\varepsilon},   
    \end{align*}
    where we used \eqref{eq:lem_CLT_condition_1} in the last step.
    Now \eqref{eq:lem_CLT_condition_3} yields
    \[ \limsup_{n\to\infty} |\Psi_n(t)| \leq 0 + 2|t|/\varepsilon. \]
    Since $\varepsilon $ can be chosen arbitrarily large, it follows that $\Psi_n(t) \to 0$, for every $t\in\R$, so $X_n \xrightarrow{D} 0$ by Lévy's continuity theorem.
\end{proof}

For further background material on infinitely divisible distributions we refer to the monographs \cite{GK54,Sato} as well as to Chapter 7 of \cite{kallenberg}.

\section{Estimates for beta functions}\label{sec:Beta}

The \textit{Beta function} is defined as
\begin{equation}\label{eq:beta_def}
    B(p,q) := \frac{\Gamma(p)\Gamma(q)}{\Gamma(p+q)}, \quad p,q >0,
\end{equation}
where $\Gamma$ is the usual Gamma function.
It is readily seen from the integral representation
$B(p,q) = \int_0^1 t^{p-1} (1-t)^{q-1} \dint t$ 
that \(B(p,q)\) is decreasing in both arguments \(p,q\).
It follows in particular that
\begin{equation*}
    B(p,q) \leq B({\textstyle\frac{1}{2}},{\textstyle\frac{1}{2}}) = \frac{\Gamma(\frac{1}{2})^2}{\Gamma(1)} = \pi, \quad p,q \ge \frac{1}{2},
\end{equation*}
where we used \eqref{eq:beta_def}, as well as the fact that \(\Gamma(1/2) = \sqrt{\pi}\) and \(\Gamma(1) = 1\).
Thus,
\begin{equation}\label{eq:beta_bound_pi}
    \frac{\Gamma(p+q)}{\Gamma(p)} = \frac{1}{B(p,q)} \Gamma(q) \geq \frac{1}{\pi} \Gamma(q), \quad p,q \ge \frac{1}{2}.
\end{equation}
The inequality
\begin{equation}\label{eq:gautschi}
    z^{1-t} \le \frac{\Gamma(z+1)}{\Gamma(z+t)} 
    , \quad z\ge 0,\, t \in [0,1], 
\end{equation}
is due to Wendel, see \cite[Equation (4)]{Wendel} (for $z=0$ and $t=1$ relation \eqref{eq:gautschi} holds with $0^0:=1$ and for $z=t=0$ it holds with $\Gamma(0)^{-1}=0$).
\begin{lemma}\label{lem:gamma_ratio}
\begin{enumerate}[label=\emph{(\alph*)}]
\item  If $p \ge 1$ and $q \geq 0$, then $\Gamma(p+q) \geq \Gamma(p) (p-1)^q$.
\item  If $p,q > 0$, then $\Gamma(p+q) \leq \Gamma(p)(p+q)^q$.
\item  If $q \geq 0$ is fixed, then $\Gamma(p+q) \sim \Gamma(p)p^q, \quad p \to \infty$.
\end{enumerate}
\end{lemma}
\begin{proof} (a) Let $q=n+s$ with $n\in\N_0$ and $s\in [0,1)$. Repeated application of \(\Gamma(x+1) = x \Gamma(x)\), $x>0$, yields 
$\Gamma(p+q)\ge \Gamma(p+s)(p+s)^n$. From \eqref{eq:gautschi} with $z=p+s-1\ge 0$ and $t=1-s$, we obtain $\Gamma(p+s)\ge \Gamma(p)(p+s-1)^s$. Hence, $\Gamma(p+q)\ge \Gamma(p)(p+s)^n(p+s-1)^s\ge \Gamma(p)(p-1)^{n+s}=\Gamma(p)(p-1)^q$.

    (b) Let $q=n+s$ with $n\in\N_0$ and $s\in (0,1]$. Repeated application of \(\Gamma(x+1) = x \Gamma(x)\), $x>0$, and  \eqref{eq:gautschi} with $z=p$ and $t=s$ yield $\Gamma(p+q)\le (p+q)^n\Gamma(p+s)\le (p+q)^{n+s}\Gamma(p+1)p^{-1}=\Gamma(p)(p+q)^{q}$.

    (c) The last claim follows directly from parts (a) and (b).
\end{proof}
\noindent Finally, we recall from \cite[Section 12.33]{WW} that 
\begin{equation}\label{eq:Stirlingapp}
\sqrt{\frac{2\pi}{z}}\left(\frac{z}{e}\right)^z\le \Gamma(z)\le e^{\frac{1}{12z}}\sqrt{\frac{2\pi}{z}}\left(\frac{z}{e}\right)^z,\quad  z\ge 1,
\end{equation}
which is the lower and upper  \textit{Stirling approximation} for the Gamma function.

\subsection{Incomplete Beta function and Beta distribution}

For \(p,q > 0\) and \(x \in [0,1]\), we write
\[B_x(p,q) \coloneqq \int_0^x u^{p-1} (1-u)^{q-1} \,\dint u \]
for the \textit{incomplete Beta function} and
$$
I_x(p,q) \coloneqq I(p,q;x)\coloneqq B_x(p,q)/B(p,q)
$$
for the \textit{regularized} (or normalized) \textit{incomplete Beta function}.
We extend the domain of $I(p,q;\cdot)$ to the whole of $\R$ by setting $I(p,q;x) \coloneqq 0$ for $x < 0$ and $I(p,q;x) \coloneqq 1$ for $x > 1$.
The probability distribution \(\Beta(p,q)\) on $\R$ with cumulative distribution function \(I(p,q;\, \cdot\, )\) is called the \emph{Beta distribution} with parameters \(p,q\).
The mean and variance of \(\Beta(p,q)\) are given by \(\mu := \frac{p}{p+q}\) and \(\sigma^2 := \frac{pq}{(p+q)^2(p+q+1)}\), respectively.

\begin{lemma}\label{lem:chebyshef}
    Let \(p,q > 0\) and \(x \in \R\), and
    write \(\mu = \frac{p}{p+q}\).
    Then $I_x(p,q) \leq \frac{1}{p(\frac{x}{\mu}-1)^2}$ for $x < \mu$ and $I_x(p,q) \geq 1 - \frac{1}{p(\frac{x}{\mu}-1)^2}$ for $x > \mu$.
\end{lemma}
\begin{proof}
    Let \(Y\) be a random variable with distribution $\Beta(p,q)$.
    Then \(\sigma^2/\mu^2 \leq 1/p\), as follows from the expressions for $\mu$ and $\sigma^2$ above.
    For \(x < \mu\), the Chebyshef inequality then yields
    \[ I_x(p,q) = \P(Y \leq x) \leq \P(|Y-\mu| \geq |x-\mu|) \leq \frac{\sigma^2}{(x-\mu)^2}
    = \frac{\sigma^2}{\mu^2} \frac{1}{(\frac{x}{\mu}-1)^2}  \leq \frac{1}{p(\frac{x}{\mu}-1)^2}.\]
    Likewise, for \(x > \mu\) we get 
    \[ 1 - I_x(p,q) = \P(Y > x) \leq \P(|Y-\mu| > |x-\mu|) \leq \frac{1}{p(\frac{x}{\mu}-1)^2},\]
which concludes the argument.
\end{proof}

An application of Lemma \ref{lem:chebyshef} yields the following result, which in turn will be used below to establish Theorems \ref{thm:CLT}   and \ref{thm:rescaling_levy_measure}.

\begin{lemma}\label{lem:consequences1}
Let \((p_n)_{n\geq 1}, (q_n)_{n\geq 1}\) be sequences in \((0,\infty)\) and \((x_n)_{n\geq 1}\) be a sequence in $\R$.
Write \(\mu_n = p_n/(p_n+q_n)\). 
\begin{enumerate}
    \item[\emph{(a)}] If \(x_n/\mu_n \to \infty\) and \(\liminf\limits_{n\to\infty}p_n > 0\), then \(I_{x_n}(p_n,q_n) \to 1\).
    \item[\emph{(b)}] If \(\liminf\limits_{n\to\infty}x_n/\mu_n > 1\) and \(p_n \to \infty\), then \(I_{x_n}(p_n,q_n) \to 1\).
    \item[\emph{(c)}] If \(\limsup\limits_{n\to\infty}x_n/\mu_n < 1\) and \(p_n \to \infty\), then \(I_{x_n}(p_n,q_n) \to 0\).
\end{enumerate}
\end{lemma}

\section{Proofs}\label{sec:proofs}

Before giving the proofs of Theorems \ref{thm:CLT} and \ref{thm:rescaling_levy_measure}, we start with some preparations.

We call a pair of integers $(d,k)$ \textit{admissible} if $1\le k\le d-1$ and $2k>d+1$ (a sequence of such pairs is admissible if each pair in the sequence is admissible). 
Recall that for admissible $(d,k)$, the characteristic function of $Z_{d,k}$ is given by \eqref{eq:char_fct_Z}.
The variance $\sigma^2_{d,k} = \V(Z_{d,k}) = \int_\R x^2 \,\nu_{d,k}(\dint x)$ is finite, see \cite[Remark 1.5(ii)]{BHT23}.
Let $\nu^*_{d,k}$ denote the Lévy measure of $Z^*_{d,k} \coloneqq Z_{d,k} / \sigma_{d,k}$, which is given by $\nu^*_{d,k}(A) = \nu_{d,k}(\sigma_{d,k} A)$ for measurable $A\subset\R$.
Then $Z^*_{d,k}$ satisfies condition \eqref{eq:lem_CLT_condition_1}.
Define
\[ J(d,k,\varepsilon) \coloneqq \int_{\{|x| > \varepsilon\}} x^2 \,\nu^*_{d,k}(\dint x) = \int_\varepsilon^\infty  x^2 \,\nu^*_{d,k}(\dint x).\]
The following is an immediate consequence of Lemma \ref{lem:condition_CLT_alternative}.

\begin{corollary}\label{cor:CLT}
    Let $(d_n,k_n)_{n\geq 1}$ be an admissible sequence. If $J(d_n,k_n,\varepsilon) \to 0$ for all $\varepsilon>0$, then $Z^*_{d_n,k_n}\xrightarrow{D}\mathcal{N}(0,1)$, as $n\to \infty$. If $J(d_n,k_n,\varepsilon) \to 1$ for all $\varepsilon>0$, then $Z^*_{d_n,k_n} \xrightarrow{D} 0$, as $n\to\infty$.
\end{corollary}

Note that for any $\varepsilon > 0$,
\[ J(d,k,\varepsilon) = \int_{\{|x| > \varepsilon\}} x^2 \,\nu^*_{d,k}(\dint x) = \int_{\{|x| > \sigma_{d,k}\varepsilon\}} \frac{x^2}{\sigma_{d,k}^2} \,\nu_{d,k}(\dint x)
= \frac{\int_{\{|x| > \sigma_{d,k}\varepsilon\}} x^2 \,\nu_{d,k}(\dint x)}{\int_\R x^2 \,\nu_{d,k}(\dint x)}. \]
Using \eqref{eq:levy_measure_Z} and substituting $u = x^\frac{2}{k-1}$, we can then express $J(d,k,\varepsilon)$ as
\begin{align} \frac{\int_{\{\sigma_{d,k}\varepsilon<x<1\}} x^{1-\frac{d-1}{k-1}}
	\bigl(1-x^{\frac{2}{k-1}}\bigr)^{\frac{d-k}{2}-1} \dint x}{\int_0^1 x^{1-\frac{d-1}{k-1}}
	\bigl(1-x^{\frac{2}{k-1}}\bigr)^{\frac{d-k-2}{2}} \dint x}\nonumber
    &= \frac{\int_{\{(\sigma_{d,k}\varepsilon)^\frac{2}{k-1} < u < 1\}} u^{\frac{2k-d-1}{2}-1} \, (1-u)^{\frac{d-k}{2}-1}\dint u}{{\int_0^1 u^{\frac{2k-d-1}{2}-1} \, (1-u)^{\frac{d-k}{2}-1}\dint u}}\nonumber\\
&= 1 - I\left({\textstyle\frac{2k-d-1}{2}}, {\textstyle\frac{d-k}{2}}; (\sigma_{d,k}\varepsilon)^\frac{2}{k-1}\right) \label{eq:J_identity_1}\\
&= I\left({\textstyle\frac{d-k}{2}}, {\textstyle\frac{2k-d-1}{2}}; 1- (\sigma_{d,k}\varepsilon)^\frac{2}{k-1}\right)\label{eq:J_identity_2},
\end{align}
where we used the fact that $I(p,q;x) + I(q,p;1-x) = 1$ for $p,q > 0$ and $x \in \R$ in the last line.

Note that by the same approach, one obtains
\begin{equation}
\label{eq:sigma}
    \sigma^2_{d,k}
    = \frac{\omega_{d-k}}{2} B\left(\textstyle\frac{2k-d-1}{2},\frac{d-k}{2}\right) 
    = \pi^\frac{d-k}{2} \frac{\Gamma\left(\frac{2k-d-1}{2}\right)}{\Gamma\left(\frac{2k-d-1}{2} + \frac{d-k}{2}\right)}
    = \pi^\frac{d-k}{2} \frac{\Gamma\left(\frac{2k-d-1}{2}\right)}{\Gamma\left(\frac{k-1}{2}\right)},
\end{equation}
where we used \eqref{eq:beta_def} for the second equality. The next lemma will ensure that $\sigma_{d_n,k_n}\varepsilon$, and hence the last argument of $I(\,\cdot\,,\,\cdot\,;\,\cdot\,)$ in \eqref{eq:J_identity_1} and \eqref{eq:J_identity_2}, lies between zero and one for large enough $n$. 
We believe this to be noteworthy, even though it is not needed for the proofs of Theorems \ref{thm:CLT} and \ref{thm:rescaling_levy_measure}.

\begin{lemma}\label{Le1}
   If \((d_n,k_n)_{n\geq 1}\) is an admissible sequence with $d_n \to \infty$,  then  \(\sigma_n=\sigma_{d_n,k_n} \to 0\), as $n\to\infty$.
\end{lemma}
\begin{proof}
Recall that $d_n-k_n\ge 1$ and $2k_n-d_n-1\ge 1$. Let $\varepsilon_0>0$ be given. We show that $\sigma_n^2\le \varepsilon_0$ if $n$ is large enough. In the following, we occasionally omit the index $n$. 

It follows from \eqref{eq:sigma}, \eqref{eq:beta_bound_pi} and \eqref{eq:Stirlingapp} that
    \[ \sigma^2_{d,k}
    = \pi^\frac{d-k}{2} \frac{\Gamma\left(\frac{2k-d-1}{2}\right)}{\Gamma\left(\frac{2k-d-1}{2} + \frac{d-k}{2}\right)}
    \leq  \pi^{\frac{d-k}{2}+1} \frac{1}{\Gamma\left(\frac{d-k}{2}\right)}
    \le \sqrt{d-k}\left(\frac{2e\pi}{d-k}\right)^{\frac{d-k}{2}}. \]
There is some $C_0> 1$ such that the upper bound is less or equal $\varepsilon_0$ if $d-k>C_0$. Hence $\sigma_n^2\le \varepsilon_0$ for all $n\in\N$ such that $d_n-k_n>C_0$. 

Now we consider all $n\in\N$ such that $d_n-k_n\le C_0$. For any such $n$ we have $k_n\ge d_n-C_0$ and $2k_n-d_n-1\ge d_n-2C_0-1$.
Hence, if there are infinitely many $n\in\N$ such that $d_n-k_n\le C_0$, then we can assume that $(2k_n-d_n-1)/2\ge 2$ if $n$ is large enough and $d_n-k_n\le C_0$.
For any such $n$,  the monotonicity of the Gamma function and Lemma \ref{lem:gamma_ratio} (a) yield 
\[ \sigma^2_{d,k}
    \le  \pi^\frac{C_0}{2} \frac{\Gamma\left(\frac{2k-d-1}{2}\right)}{\Gamma\left(\frac{2k-d-1}{2} + \frac{1}{2}\right)}
    \le  \pi^\frac{C_0}{2}\left(\frac{2k-d-3}{2}\right)^{-\frac{1}{2}}\le 
 \pi^\frac{C_0}{2}\left(\frac{d-2C_0-3}{2}\right)^{-\frac{1}{2}}   .
\]
If $n$ is large enough, the right-hand side is at most $\varepsilon_0$. 

Thus, if $n\in\N$ is large enough, in any case we obtain $\sigma_n^2\le \varepsilon_0$.
\end{proof}

\subsection{Proof of Theorem \ref{thm:CLT}}

Throughout this section, we consider admissible sequences $(d_n,k_n)_{n\geq 1}$.
Note that by Corollary \ref{cor:CLT}, the determination of $\lim_{n\to \infty}J(d_n,k_n,\varepsilon)$ is sufficient for proving the assertions of Theorem \ref{thm:CLT}. We start with the case where $\liminf_{n \to \infty} d_n/k_n > 1/2$. In what follows, we will frequently use the abbreviation $\sigma_n \coloneqq \sigma_{d_n,k_n}$.

\begin{proposition}\label{prop:k/d>1/2}
    Let \((d_n,k_n)_{n\geq 1}\) be an admissible sequence with \(d_n \to \infty\), as \(n\to\infty\).
    Suppose that \(\liminf\limits_{n\to\infty}k_n/d_n > 1/2\). Then \(\lim\limits_{n\to\infty} \JJ(d_n,k_n,\varepsilon) = 1\),  for any \(\varepsilon > 0\).
\end{proposition}
\begin{proof}
    Fix an $\varepsilon > 0$.
    Without loss of generality, we assume that \(\lim_{n\to \infty} k_n/d_n \eqqcolon \alpha \in (1/2,1]\) exists.
    Using \eqref{eq:J_identity_2}, we have \(J(d_n,k_n,\varepsilon)=I(p_n,q_n;x_n)\) with
    $p_n = \frac{d_n-k_n}{2}$, $q_n = \frac{2k_n-d_n-1}{2}$ and $x_n = 1 - (\sigma_n\varepsilon)^\frac{2}{k_n-1}$.
    Dropping the index \(n\) where it is unlikely to cause confusion, we write
    \begin{align*}
        x_n \frac{p_n+q_n}{p_n}
        &= (1-(\sigma\varepsilon)^\frac{2}{k-1}) \frac{k-1}{d-k}
        = \Big( 1 - \varepsilon^\frac{2}{k-1} \Big(\pi^\frac{d-k}{2}\frac{\Gamma(\frac{2k-d-1}{2})}{\Gamma(\frac{2k-d-1}{2}+\frac{d-k}{2})}\Big)^\frac{1}{k-1} \Big) \frac{k-1}{d-k},
    \end{align*}
    where we used \eqref{eq:sigma} for the last step.
    Using Lemma \ref{lem:gamma_ratio} (a)  and the fact that $\frac{2k_n-d_n-3}{2} \geq \frac{2\alpha-1}{4}d_n$ if \(n\) is large enough, we get that
    \begin{align*}
        \liminf_{n \to \infty} x_n \frac{p_n+q_n}{p_n}
        &\geq \liminf_{n\to\infty} \Big( 1 - \varepsilon^\frac{2}{k-1} \Big(\pi^\frac{d-k}{2}\frac{1}{(\frac{2k-d-3}{2})^{\frac{d-k}{2}}}\Big)^\frac{1}{k-1} \Big) \frac{k-1}{d-k}\\
        &\geq \liminf_{n\to\infty} \Big( 1 - \varepsilon^\frac{2}{k-1} \Big(\pi^\frac{d-k}{2}\frac{1}{(\frac{2\alpha-1}{4}d)^{\frac{d-k}{2}}}\Big)^\frac{1}{k-1} \Big) \frac{k-1}{d-k}\\
        &= \liminf_{n\to\infty} \Big( 1 - \Big(\varepsilon^\frac{4}{d-k} \frac{4\pi}{2\alpha-1}\frac{1}{d}\Big)^\frac{d-k}{2(k-1)} \Big) \frac{k-1}{d-k}.
    \end{align*}
    If \(\alpha < 1\), then \(p_n \to \infty\), and the above bound yields
    \[ 
    \liminf_{n\to\infty}x_n \frac{p_n+q_n}{p_n} \geq\liminf_{n\to\infty} \Big( 1 - \Big(\varepsilon^\frac{4}{d-k} \frac{4\pi}{2\alpha-1}\frac{1}{d}\Big)^\frac{d-k}{2(k-1)} \Big) \frac{k-1}{d-k} = (1 - 0) \frac{\alpha}{1-\alpha} > 1,
    \]
    which implies \(I(p_n,q_n;x_n) \to 1\) by Lemma \ref{lem:consequences1} (b).

    Let us now assume that \(\alpha = 1\).
    Note that for any \(c > 0\), we have that
    \[ \varepsilon^\frac{4}{d_n-k_n} \frac{4\pi}{2\alpha-1}\frac{1}{d_n} \leq c \]
    for all but finitely many \(n\), which gives
    \[
    \liminf_{n\to\infty} \Big( 1 - \Big(\varepsilon^\frac{4}{d-k} \frac{4\pi}{2\alpha-1}\frac{1}{d}\Big)^\frac{d-k}{2(k-1)} \Big) \frac{k-1}{d-k} \geq \liminf_{n\to\infty} \Big( 1 - c^\frac{d-k}{2(k-1)} \Big) \frac{k-1}{d-k}.
    \]
    Since \(\frac{d-k}{2(k-1)} \to \frac{1-\alpha}{2\alpha} = 0\), the latter limit is equal to
    $\lim_{t \to 0+} \frac{1 - c^t}{2t} = -\frac{\log c}{2}$.
   
    Since \(c > 0\) was arbitrary, this yields
    \[ \liminf_{n \to \infty} x_n \frac{p_n+q_n}{p_n} \geq \sup_{c>0}- \frac{\log c}{2} = \infty, \]
    which, together with $\liminf_{n\to \infty}p_n \geq 1/2 >0$, implies \(I(p_n,q_n;x_n) \to 1\) by Lemma \ref{lem:consequences1} (a).
\end{proof}

We now focus on admissible sequences with $d_n/k_n \to 1/2$.
The asymptotic behaviour then depends on the growth of $r_n = 2k_n-d_n-1$.
If $r_n$ is sufficiently small (which means that $k_n$ is sufficiently close to $d_n/2$), then $J(d_n,k_n,\varepsilon)$ tends to $0$ for any $\varepsilon > 0$.

\begin{proposition}\label{prop:k/d_to_1/2}
Let $\varepsilon>0$. Let $(d_n,k_n)_{n\geq 1}$  be an  admissible sequence such that $d_n\to\infty$ and  $k_n/d_n\to 1/2$.
Write $r_n := 2k_n-d_n-1$.
Then the following holds:
\begin{enumerate}
     \item[\emph{(a)}]  If $\limsup\limits_{n\to\infty}d_n^{-1} r_n^{d_n/k_n} < e\pi$, then \(J(d_n,k_n,\varepsilon) \to 0\) as $n\to\infty$.
     \item[\emph{(b)}]  If $\liminf\limits_{n\to\infty}d_n^{-1} r_n^{d_n/k_n} > e\pi$, then \(\JJ(d_n,k_n,\varepsilon) \to 1\) as $n\to\infty$.
\end{enumerate}
\end{proposition}

\begin{proof}
We fix  \(\varepsilon > 0\) and omit the index \(n\). Then, we notice that, as \(n \to \infty\),
\begin{align*}
\sigma^{\frac{4}{k-1}}
    &= \pi^\frac{d-k}{k-1} \frac{\Gamma\left(\frac{2k-d-1}{2}\right)^{\frac{2}{k-1}}}{\Gamma\left(\frac{k-1}{2}\right)^{\frac{2}{k-1}}}
    \sim \pi^{\frac{\frac{d}{k}-1}{1-\frac{1}{k}}} \frac{\Gamma\left(\frac{2k-d-1}{2}\right)^{\frac{2}{k-1}}}{(4\pi/(k-1))^\frac{1}{k-1}\frac{k-1}{2e}}
    \sim  \frac{2e\pi}{k} \Gamma\left({\textstyle\frac{2k-d-1}{2}}\right)^{\frac{2}{k-1}},
\end{align*}
where we used \eqref{eq:sigma} and \eqref{eq:Stirlingapp}.
While \(r = 2k-d-1\) may not tend to infinity, \eqref{eq:Stirlingapp} still guarantees that
\[ 0 < \inf_{r \in \N} \frac{\Gamma(\frac{r}{2})}{\sqrt{\frac{4\pi}{r}} (\frac{r}{2e})^\frac{r}{2}} \leq \sup_{r \in \N} \frac{\Gamma(\frac{r}{2})}{\sqrt{\frac{4\pi}{r}} (\frac{r}{2e})^\frac{r}{2}} < \infty, \]
which yields
\[ \frac{2e\pi}{k} \Gamma\left({\textstyle\frac{2k-d-1}{2}}\right)^{\frac{2}{k-1}} \sim \frac{2e\pi}{k} \left(\frac{r}{2e}\right)^\frac{r}{k-1} \sim \frac{2e\pi}{k} r^\frac{r}{k-1}, \]
where the last step holds because  \(r/(k-1)\) tends to zero.
Hence,
\begin{align}\label{eq:later}
    \sigma^{\frac{4}{k-1}}\Big(\frac{k-1}{r}\Big)^2
    \sim \frac{2e\pi}{k} r^{\frac{r}{k-1}} \Big(\frac{k}{r}\Big)^2
    = e\pi \frac{2k}{r^\frac{d-1}{k-1}}
    \sim e\pi \frac{d}{r^{\frac{d}{k}}},
\end{align}
where we used
$$ \frac{r^{\frac{d}{k}}}{r^{\frac{d-1}{k-1}}}=\Big(r^{\frac{1}{k-1}}\Big)^{1-\frac{d}{k}}\to 1^{-1}=1 $$
in the last step.
By \eqref{eq:J_identity_1}, we have \(\JJ(d_n,k_n,\varepsilon)=1 - I(p_n,q_n;x_n)\) with $p_n = \frac{r_n}{2}$, $q_n = \frac{d_n-k_n}{2}$ and $x_n = (\sigma_n\varepsilon)^\frac{2}{k_n-1}$.

To prove (a), suppose that  $\limsup_{n\to\infty}(d_n^{-1} r_n^{d_n/k_n}) < e\pi$.
Without loss of generality, we may assume that either the sequence \((r_n)_{n\geq 1}\) is bounded or \(r_n \to \infty\). In the former case, we observe that \(\liminf_{n\to\infty} p_n \geq 1/2\) since \(r_n \geq 1\) and further, using \eqref{eq:later}, 
\[ \Big(x_n\frac{p_n+q_n}{p_n}\Big)^2 
    = \sigma^{\frac{4}{k-1}}\Big(\frac{k-1}{r}\Big)^2
    \sim e\pi \frac{d}{r^{\frac{d}{k}}} \geq e\pi \frac{d}{(\max_n r_n)^2} \to \infty. \]
Lemma \ref{lem:consequences1} (a) thus yields \(\JJ(d_n,k_n,\varepsilon) \to 0\). In the latter case, \(p_n = r_n/2 \to \infty\) and \eqref{eq:later} gives
\[ \liminf_{n \to \infty} \Big(x_n\frac{p_n+q_n}{p_n}\Big)^2 
    = \liminf_{n\to\infty} e\pi  \frac{d}{r^\frac{d}{k}} > 1, \]
hence Lemma \ref{lem:consequences1} (b) yields \(\JJ(d_n,k_n,\varepsilon) \to 0\).
Together with the previous consideration, this proves the first part of the proposition.

For part (b) suppose that $\liminf_{n\to\infty}(d_n^{-1} r_n^{d_n/k_n}) > e\pi$.
Arguing as in the previous case then gives \(\limsup_{n \to \infty} x_n\frac{p_n+q_n}{p_n} < 1\).
Since \(r_n^2 \geq r_n^\frac{d_n}{k_n} \geq e \pi d_n\) for all but finitely many \(n\), we further have \(p_n = r_n/2 \to \infty\).
We can thus apply Lemma \ref{lem:consequences1} (c) to obtain $\JJ(d_n,k_n,\varepsilon) \to 1$.
\end{proof}

Theorem \ref{thm:CLT} now follows directly by combining Proposition \ref{prop:k/d>1/2}, Proposition \ref{prop:k/d_to_1/2} and Corollary \ref{cor:CLT}.

\subsection{Proof of Theorem \ref{thm:rescaling_levy_measure}}

For $k<d$ and $2k>d+1$ let $f_{d,k}$ denote the density of the Lévy measure in \eqref{eq:levy_measure_Z}.
Using \eqref{eq:sigma} and $\omega_{d-k}=2\pi^{\frac{d-k}{2}}/\Gamma(\frac{d-k}{2})$, one obtains for $x \in (0,1)$ that 
\begin{equation}\label{eq:representation_f_normalized}
\frac{f_{d,k}(x)}{\sigma_{d,k}^2} = \frac{2}{k-1} x^{-1-\frac{d-1}{k-1}} (1-x^\frac{2}{k-1})^{\frac{d-k}{2}-1} \frac{\Gamma\left(\frac{2k-d-1}{2}+\frac{d-k}{2}\right)}{\Gamma\left(\frac{2k-d-1}{2}\right)\Gamma\left(\frac{d-k}{2}\right)}. 
\end{equation}
In the following, we consider an admissible sequence $(d_n,k_n)_{n\geq 1}$ as well as random variables
 $\widetilde Z_{n}$, $n\in\N$,  with respective Lévy triplets $(0,0,\nu_{d_n,k_n}/\sigma^2_{d_n,k_n})$.

From now on, we will abbreviate $f_n \coloneqq f_{d_n,k_n}$, $\sigma_n \coloneqq \sigma_{d_n,k_n}$ and $\nu_n\coloneqq\nu_{d_n,k_n}$ to simplify the notation.
Sometimes the index $n$ will be omitted from $d_n$ and $k_n$ for the sake of readability.

Let us first assume that $d_n-k_n = b \in \N$ is constant.
We show that $\widetilde Z_{n}$ converges to $\widetilde Z ^{(b)}$ in distribution.

\smallskip
\emph{Step 1: Pointwise limit of the rescaled densities.}
Let $x \in (0,1)$ be fixed.
Since $k_n \to \infty$ as $n \to \infty$, it follows by l'Hôpital's rule that
$1 - x^\frac{2}{k_n-1} \sim (-2\log x)/(k_n-1)$, as $n \to \infty$. 
Using these observations as well as Lemma \ref{lem:gamma_ratio} (c), we obtain
\[ \frac{f_n(x)}{\sigma^2_n} \sim \frac{2}{k_n-1} x^{-2} \left(-\frac{2 \log x}{k_n-1}\right)^{\frac{b}{2}-1} \frac{\left(\frac{2k_n-d_n-1}{2}\right)^\frac{b}{2}}{\Gamma(\frac{b}{2})}. \]
Since $k_n-1 = d_n -b-1 \sim d_n$ and $2k_n-d_n-1 = d_n-2b-1 \sim d_n$, it follows that
\[ \frac{f_n(x)}{\sigma^2_n} \to \Gamma({\textstyle\frac{b}{2}})^{-1} x^{-2} (-\log x)^{\frac{b-2}{2}}, \quad n \to \infty, \]
so $f_n/\sigma^2_n$ converges pointwise to the density of the Lévy measure $\widetilde\nu^{(b)}$.

\smallskip
\emph{Step 2: Integrable domination.}
Let $g_t(x):=e^{itx}-1-itx$, for $t,x \in \R$.
We construct, for fixed $t \in \R$, an integrable dominating function for
\( |g_t(x)|\, f_n(x)/\sigma^2_n\), \(x\in (0,1)\).
Since $|g_t(x)| \leq t^2 x^2$ by \eqref{eq:bound_taylor}, we get using \eqref{eq:representation_f_normalized} that
\[ |g_t(x)|\frac{f_n(x)}{\sigma^2_n}
\leq \frac{2t^2}{k-1} x^{1-\frac{d-1}{k-1}} (1-x^\frac{2}{k-1})^{\frac{b}{2}-1} \frac{\Gamma\left(\frac{2k-d-1}{2} + \frac{b}{2}\right)}{\Gamma\left(\frac{2k-d-1}{2}\right)\Gamma\left(\frac{b}{2}\right)}.  \]
Let us first assume that $b \geq 2$.
Since $\frac{d_n-1}{k_n-1} = \frac{d_n-1}{d_n-b-1} \to 1$, as $n \to \infty$, it holds for large $n$ that
\[ x^{1-\frac{d_n-1}{k_n-1}} \leq x^{-\frac{1}{2}}, \quad x \in (0,1). \]
From $1 + y \leq e^y$ with $y = (2 \log x)/(k-1)$, we get
\[ 1-x^\frac{2}{k-1} \leq -\frac{2 \log x}{k-1}, \quad x \in (0,1). \]
Using these bounds as well as Lemma \ref{lem:gamma_ratio} (b), we obtain
\[ |g_t(x)|\frac{f_n(x)}{\sigma^2_n}
\leq \frac{2 t^2}{k-1} x^{-\frac{1}{2}} \left(-\frac{2 \log x}{k-1}\right)^{\frac{b}{2}-1} \frac{\left(\frac{k-1}{2}\right)^\frac{b}{2}}{\Gamma(\frac{b}{2})}
= \frac{t^2}{\Gamma(\frac{b}{2})} x^{-\frac{1}{2}} (-\log x)^{\frac{b}{2}-1}, \quad x \in (0,1), \]
for all sufficiently large $n\in\N$.
The integrability of the right-hand side follows after substituting $x=e^{-u}$.

Now assume that $b=1$.
For $x \in (0,1]$ and $\alpha \in [0,1]$, the mean value theorem implies that $\alpha(1-x) \leq 1-x^\alpha$.
Applying this with $\alpha = 2/(k_n-1)$, which is in $[0,1]$ since $k_n \geq 3$, gives
\[ (1-x^\frac{2}{k_n-1})^{-\frac{1}{2}} \leq \Big( \frac{2}{k_n-1}(1-x) \Big)^{-\frac{1}{2}}.  \]
Bounding the other terms of $|g_t(x)|\,f_n(x)/\sigma_n^2$ as above yields
\[ |g_t(x)|\frac{f_n(x)}{\sigma^2_n}
\leq \frac{2 t^2}{k-1} x^{-\frac{1}{2}} \left(\frac{2}{k-1}(1-x)\right)^{-\frac{1}{2}} \frac{\left(\frac{k-1}{2}\right)^\frac{1}{2}}{\Gamma(\frac{1}{2})}
= \frac{t^2}{\Gamma(\frac{1}{2})} x^{-\frac{1}{2}} (1-x)^{-\frac{1}{2}}, \quad x \in (0,1), \]
for all but finitely many $n \in \N$. This function is clearly integrable.

\smallskip
\emph{Step 3: Convergence of exponents and random variables.}
Define
\[
\Psi_n(t) := \frac{1}{\sigma^2_{n}} \int_{(0,1)} g_t(x)\,\nu_{n}(\dint x),
\quad 
\Psi^{(b)}(t) := \int_{(0,1)} g_t(x)\,\widetilde\nu^{(b)}(\dint x),\quad t \in \R,
\]
with $g_t(x)=e^{itx}-1-itx$ as above.
By Step~1, $f_n(x)/\sigma^2_n\to \Gamma(b/2)^{-1} x^{-2}(-\log x)^{(b-2)/2}$ for each $x\in(0,1)$, as $n\to\infty$, and by Step~2, for each fixed $t\in\R$, the integrands $x \mapsto g_t(x) f_n(x) / \sigma^2_n$ are dominated by an integrable function independent of $n$, provided that $n$ is large enough.
Hence, the dominated convergence theorem yields for each $t \in \R$ that $\Psi_n(t) \to \Psi^{(b)}(t)\in (0,\infty)$, as $n\to\infty$.

Let $\widetilde Z^{(b)}$ be an infinitely divisible random variable with Lévy triplet $(0,0,\widetilde\nu^{(b)})$.
(The fact that $\widetilde\nu^{(b)}$ is a valid Lévy measure, i.e., that $\int_\R (1\wedge x^2)\, \widetilde\nu^{(b)}(\dint x) < \infty$, follows from \eqref{eq:cumulants} below.)
It then follows for the characteristic functions of $\widetilde Z_n$ and $\widetilde Z^{(b)}$ that
\[
\varphi_{\widetilde Z_n}(t) = \exp(\Psi_n(t)) \ \to \ \exp(\Psi^{(b)}(t)) = \varphi_{\widetilde Z^{(b)}}(t) \qquad \text{for each } t\in\R,
\]
so $\widetilde Z_n \xrightarrow{D} \widetilde Z^{(b)}$ by Lévy's continuity theorem.

For $m\geq 2$, the $m$-th order cumulant of $\widetilde Z^{(b)}$ is given by
\begin{equation}\label{eq:cumulants}
    \kappa_m\bigl(\widetilde Z^{(b)}\bigr)
    =\int_{(0,1)} x^m\,\widetilde\nu^{(b)}(\mathrm{d}x)
    = \frac{1}{\Gamma(\frac{b}{2})}\int_0^1 x^{m-2}\,(-\log x)^{\frac{b-2}{2}}\,\mathrm{d}x
    =\Big(\frac{1}{m-1}\Big)^{\frac{b}{2}},
\end{equation}
since $\int_0^1 x^{p-1}(-\log x)^\alpha\,\mathrm{d}x=\Gamma(\alpha+1)/p^{\alpha+1}$ 
(substitute $x=e^{-u/p}$). 
In particular, the random variable $\widetilde Z^{(b)}$ has finite moments of all orders and $\V(\widetilde Z^{(b)})=1$. This concludes the proof of the first part of Theorem \ref{thm:rescaling_levy_measure}.

Let us now assume that $d_n-k_n \to \infty$, as $n \to \infty$.
Write $\widetilde \nu_{n} \coloneqq \frac{1}{\sigma^2_n} \nu_n$.
For arbitrary $\varepsilon > 0$, it follows from \eqref{eq:representation_f_normalized} that
\begin{align*}
    \int_{\{|x| > \varepsilon\}} x^2 \,\widetilde\nu_n(\dint x) 
    &= B\left({\textstyle\frac{2k-d-1}{2}},{\textstyle\frac{d-k}{2}}\right)^{-1} \frac{2}{k-1} \int_{\{\varepsilon < x < 1\}} x^{1-\frac{d-1}{k-1}} (1-x^\frac{2}{k-1})^{\frac{d-k}{2}-1} \,\dint x\\
    &= I\left({\textstyle\frac{d-k}{2}},{\textstyle\frac{2k-d-1}{2}};1-\varepsilon^{\frac{2}{k-1}}\right),
\end{align*}
where the second equality follows by adapting the arguments leading to \eqref{eq:J_identity_2}.
We now apply Lemma \ref{lem:consequences1} (c) with $p_n = \frac{d_n-k_n}{2}$, $q_n = \frac{2k_n-d_n-1}{2}$ and $x_n = 1 - \varepsilon^\frac{2}{k_n-1}$.
Note that $p_n \to \infty$ by our assumptions on $d_n-k_n$.
Furthermore, we have
\[ \frac{x_n}{\mu_n} = \frac{k_n-1}{d_n-k_n}(1-\varepsilon^\frac{2}{k_n-1})  \sim -\frac{2}{d_n-k_n}\log \varepsilon  \sim 0, \]
where second step follows from $\lim_{s \to 0+} (1-\varepsilon^s)/s = -\log \varepsilon$ and the last one from $d_n-k_n \to \infty$.
Hence, the conditions of Lemma \ref{lem:consequences1} (c) are met, so
\[ \int_{\{|x|>\varepsilon\}} x^2 \, \widetilde\nu_n(\dint x) \to 0, \quad n \to \infty. \]
Since $\varepsilon > 0$ was arbitrary and $\int_\R x^2 \,\widetilde\nu_n(\dint x) = 1$ by construction, Lemma \ref{lem:condition_CLT_alternative} (a) yields that $\widetilde Z_n \xrightarrow{D} \mathcal{N}(0,1)$.
This concludes the proof of the second part. \hfill\qed

\subsection*{Acknowledgement}
The authors were supported by the German Research Foundation (DFG) via SPP 2265 \textit{Random Geometric Systems}.

\end{document}